\documentclass[11pt,reqno]{amsart}
\usepackage{amsmath,amssymb,xypic,mathrsfs}

\usepackage{tikz-cd}

\numberwithin{equation}{section}

\theoremstyle{plain}
				
\newtheorem{theorem}{Theorem}[section]				
\newtheorem{proposition}[theorem]{Proposition}		
\newtheorem{corollary}[theorem]{Corollary}
\newtheorem{lemma}[theorem]{Lemma}
\newtheorem{question}[theorem]{Question}

\theoremstyle{definition}
\newtheorem{definition}[theorem]{Definition}

\newcommand{\afrak}{\mathfrak a}
\newcommand{\dfrak}{\mathfrak d}
\newcommand{\gfrak}{\mathfrak g}
\newcommand{\lfrak}{\mathfrak l}
\newcommand{\ofrak}{\mathfrak o}
\newcommand{\sfrak}{\mathfrak s}
\newcommand{\SL}{\mathsf{SL}}
\newcommand{\SO}{\mathsf{SO}}
\newcommand{\Spin}{\mathsf{Spin}}
\DeclareMathOperator{\ad}{ad}
\DeclareMathOperator{\Ad}{Ad}

\begin{document}

\title[Torsors on moduli spaces of principal $G$-bundles]{Torsors on moduli spaces of principal $G$-bundles on 
curves}
	
\author[I. Biswas]{Indranil Biswas}

\address{Department of Mathematics, Shiv Nadar University, NH91, Tehsil
Dadri, Greater Noida, Uttar Pradesh 201314, India}

\email{indranil.biswas@snu.edu.in, indranil29@gmail.com}

\author[S. Mukhopadhyay]{Swarnava Mukhopadhyay}

\address{School of Mathematics, Tata Institute of Fundamental Research,
Homi Bhabha Road, Mumbai 400005, India}

\email{swarnava@math.tifr.res.in}
	
\thanks{I.B. is supported in part by a J.C. Bose fellowship (JBR/2023/000003) and S.M. by DAE, India
under project no. 1303/3/2019/R\&D/IIDAE/13820.}



\begin{abstract}
Let $G$ be a semisimple complex algebraic group with a simple Lie algebra $\gfrak$, and let
$\mathcal{M}^0_{G}$ denote the moduli stack of topologically trivial stable $G$-bundles on a
smooth projective curve $C$. Fix a theta characteristic $\kappa$ on $C$ which is
even in case $\dim{\gfrak}$ is odd. We show that there is a nonempty Zariski open substack ${\mathcal U}_\kappa$
of $\mathcal{M}^0_{G}$ such that $H^i(C,\, \text{ad}(E_G)\otimes\kappa) \,=\, 0$, $i\,=\, 1,\, 2$,
for all $E_G\,\in\, {\mathcal U}_\kappa$. It is shown that any such $E_G$ has a canonical connection.
It is also shown that the tangent bundle $T{U}_\kappa$ has a natural splitting, where $U_{\kappa}$ is
 the restriction of $\mathcal{U}_{\kappa}$ to the semi-stable locus. We also produce
an isomorphism between two naturally occurring $\Omega^1_{{M}^{rs}_{G}}$--torsors on the moduli space of regularly stable ${M}^{rs}_{G}$. 
\end{abstract}

\maketitle

\section{Introduction}

Let $C$ be a smooth complex projective curve of genus $g$, with $g\, \geq\, 2$, and let $G$ be a complex semisimple affine algebraic
group with a simple Lie algebra $\gfrak$. Denote by $\mathcal{M}^0_{G}$ the moduli stack of topologically trivial stable $G$-bundles on
$C$. Fix a theta characteristic $\kappa$ on $C$. We assume that $\kappa$ is even when $\dim \gfrak$ is odd. We prove that the following
(see Corollary \ref{cor:nonzero2}):
\begin{theorem}\label{thm:main1}
{There is a nonempty Zariski open substack ${\mathcal U}_\kappa$
of $\mathcal{M}^0_{G}$ such that
$$
H^0(C,\, {\rm ad}(E_G)\otimes\kappa) \,=\, 0\,=\, H^1(C,\, {\rm ad}(E_G)\otimes\kappa)
$$
for all $E_G\,\in\, {\mathcal U}_\kappa$.}
\end{theorem}
Using the above result (see Corollary \ref{cor:nonzero2}) the following is proved (see Theorem \ref{th1}). This theorem generalizes earlier results of \cite{BH}, \cite{BB} for $G\,=\, \text{SL}(n,{\mathbb C})$:
\begin{theorem}
\textit{Any principal $G$--bundle $E_G\, \in\, {\mathcal U}_\kappa$. has a natural algebraic connection.}
\end{theorem}

Let ${M}^{rs}_{G}$ (respectively, ${M}^{0,ss}_{G}$) denote the locus of regularly stable (respectively, semi-stable) principal $G$-bundles
which are topologically trivial. We note that
${M}^{rs}_{G}$ is the smooth locus of ${M}^{0,ss}_{G}$ except the only case where $g\,=\,2$ and $G\,=\, \text{SL}(2,
{\mathbb C})$.

The tangent bundle $T({\mathcal U}_\kappa\bigcap {M}^{rs}_{G})$ of ${\mathcal U}_\kappa\bigcap {M}^{rs}_{G}$ (the intersection is
happening in the moduli stack of regularly stable bundles and then is projected to the moduli space)
decomposes into a direct sum of vector bundles. This
raises the natural question whether distributions on ${\mathcal U}_\kappa\bigcap {M}^{rs}_{G}$ given by these
direct summands are integrable. (See Question \ref{q1}.)

Now we focus our attention to two natural torsors on ${\mathcal U}_\kappa\bigcap {M}^{rs}_{G}$ for its cotangent bundle
$T^*({\mathcal U}_\kappa\bigcap {M}^{rs}_{G})$. The first $T^*({\mathcal U}_\kappa\bigcap {M}^{rs}_{G})$--torsor is
defined by the moduli space of connections ${\mathcal C}_G$. More precisely, ${\mathcal C}_G$ is the moduli space of
pairs of the form $(E_G,\, D)$, where $E_G\, \in\, {\mathcal U}_\kappa\bigcap {M}^{rs}_{G}$ and $D$ is an
algebraic connection on $E_G$. 

Fix an ample line bundle $L$ on ${\mathcal U}_\kappa\bigcap {M}^{rs}_{G}$. The second
$T^*({\mathcal U}_\kappa\bigcap {M}^{rs}_{G})$--torsor is given by the sheaf of connections ${\rm Conn}(L)$ on $L$.
It is shown that these two $T^*({\mathcal U}_\kappa\bigcap {M}^{rs}_{G})$--torsors are isomorphic
(see Section \ref{sec-tor}). This was proved earlier in \cite{BH} for $G\,=\, \text{SL}(n,{\mathbb C})$.

We end the introduction by briefly commenting on the organization of the paper. We start by recalling the notion of conformal embeddings of 
semisimple Lie algebras and the branching rule of the affine Lie algebras arising from the adjoint representations. Conformal embedding are 
special as they force finiteness of the branching rules of some infinite dimensional representation of affine Kac-Moody algebras. Next we 
discuss functorial maps between spaces of conformal blocks induced by conformal embeddings and use the identifications of conformal blocks 
with global sections of line bundles on $\mathcal{M}_G$ to prove Theorem \ref{thm:main1}. Theorem \ref{th1} is established
in Section \ref{se5}. Section \ref{se6} shows the above mentioned isomorphism of torsors. In Section \ref{se7} the decomposition of the
tangent bundle is constructed and Question \ref{q1} is posed.

\section{Conformal Embedding and adjoint representations}

In this section, we first recall the notion of conformal embedding for affine Lie algebra and consider a particular case of a conformal 
embedding given by the adjoint representation. We also recall the branching rule for this case.

\subsection{Conformal Embedding}

Let $\gfrak$ be a simple finite dimensional Lie algebra over $\mathbb{C}$. Consider
the corresponding untwisted affine Lie algebra $$\widehat{\gfrak}
\,\,:=\,\,\gfrak\otimes \mathbb{C}((\xi))\oplus \mathbb{C}c,$$ where $c$ is a central element. Fix
a Cartan subalgebra of $\gfrak$ and also a Borel subalgebra containing the Cartan
subalgebra. Let $(\ , \ )$ be
the normalized Killing form on $\gfrak$ such that $(\theta,\, \theta)\,=\,2$, where $\theta$ is the highest root
of $\gfrak$, and $P_+$ denotes the set of dominant integral weights. For any positive integer $\ell$, consider the set 
\begin{equation}
P_{\ell}(\gfrak)\,\,:=\,\, \{ \lambda \in P_+\,\,\mid\,\, (\lambda, \theta)\,\leq\, \ell \}\,\subsetneq\, P_+.
\end{equation}
The set $P_{\ell}(\gfrak)$ parameterizes irreducible, integrable representations of the Lie
algebra $\widehat{\gfrak}$. Let $\lambda\,\in\, P_{\ell}(\gfrak)$; then the corresponding integrable
representation will be denoted by $\mathcal{H}_{\lambda,\ell}(\gfrak)$. We will often drop $\ell$ and
$\gfrak$ from the notation of integrable highest weight modules when the context is evident. 
They satisfy the following properties: 
\begin{itemize}
	\item $\mathcal{H}_{\lambda}$ is graded.
	\item The finite dimensional $\gfrak$ module $V_{\lambda}$ is contained in $\mathcal{H}_{\lambda}$ as the degree zero part.
	\item The scalar $c$ acts on $\mathcal{H}_{\lambda}$ as multiplication by the integer $\ell$. 
\end{itemize}

Now consider a orthonormal basis $\{J^1,\,\cdots,\, J^{\dim \gfrak}\}$ of the Lie algebra $\gfrak$. The 
Sugawara construction $$L^{\gfrak}_{n,\ell}\,\,:=\,\, \frac{1}{2(\ell+h^{\vee}(\gfrak))}\sum_{m\in 
\mathbb{Z}}\sum_{a=1}^{\dim \gfrak}{}^{\circ}_{\circ}J^a(m)J^a(n-m){}^{\circ}_{\circ}$$ gives an action of the 
Virasoro algebra on $\mathcal{H}_{\lambda}$ at level $\ell$. Here $h^{\vee}(\gfrak)$ denotes the dual Coxeter 
number of $\gfrak$ and $X(m)\,:=\,X\otimes \xi^m$ for any $X\,\in\, \gfrak$. We recall the definition of conformal 
embedding from Kac-Wakimoto \cite[p.~210, Section 4.2]{KacWakimoto:88}.

\begin{definition}\label{def:conformal}
An embedding $\varphi\,:\, \gfrak_1 \,\longrightarrow\, \gfrak_2$ of simple Lie algebras
is conformal if the following equality holds
\begin{equation}\label{eqn:centralcharge}
\frac{d_{\varphi}\dim \gfrak_1}{d_{\varphi}+ h^{\vee}(\gfrak_1)}
\,\,=\,\,\frac{\dim \gfrak_2}{1+ h^{\vee}(\gfrak_2)},
\end{equation}
where $d_{\varphi}$ is the ratio of the normalized Killing form of
the embedding $\varphi$ which is also known as the Dynkin index.
\end{definition}

A key feature of conformal embedding is the equality of Virasoro operators 
$L_{n,d_{\varphi}}^{\gfrak_1}\,=\,L_{n,1}^{\gfrak_2}$ as operators on level one integrable 
representation of $\widehat{\gfrak}_2$ under the obvious restriction via $\varphi$ 
(see \cite[p.~201, Corollary 3.2.1--3.2.2]{KacWakimoto:88}). Conformal embedding of semi-simple
Lie algebras into simple Lie algebras has been classified by Schellekens--Warner \cite{SW} and
independently by Bais--Bouwknegt \cite{BaisBouwknegt:87}

\subsection{The adjoint representation}

Let $\gfrak$ be a simple Lie algebra. Consider the adjoint representation 
\begin{equation}\label{eqn:ad}
\ad_{\gfrak}\,:\, \gfrak \,\hookrightarrow\, \sfrak \ofrak(\gfrak).
\end{equation} We will drop in the subscript $\gfrak$ in the notation of adjoint representation
when there is no scope for any confusion. The Dynkin index of the embedding $\ad$ is just
the dual Coxeter number $h^{\vee}(\gfrak)$ and it is easy to check that $\ad$ satisfied
the identity in \eqref{eqn:centralcharge}. Hence the embedding $\ad$ is conformal. 

Observe that the Dynkin index of the natural embedding of $\sfrak \ofrak (\gfrak )
\,\longrightarrow\, \sfrak \lfrak(\gfrak)$ is of index two and the embedding $\gfrak
\,\longrightarrow\, \sfrak \lfrak(\gfrak)$ is not conformal since it factors through
$\sfrak \ofrak (\gfrak)$. Moreover, if ${G}$ is any connected Lie group with Lie algebra
$\gfrak$, by definition of the adjoint representation the center $Z({G})\, \subset\, G$ maps to the
identity element in $\SL_{\dim \gfrak}$. However the natural map $\Spin_{\dim \gfrak}\,
\longrightarrow\, \SL_{\dim \gfrak}$ factors through $\SO_{\dim \gfrak}$, hence it is not
necessarily true that the center of the simply connected group $\widetilde{G}$ maps to
identity element under the natural lifting $\widetilde{\Ad}\,:\,\widetilde{G}\,
\longrightarrow \, \Spin_{\dim \gfrak}$ of the adjoint representation. However the image
of $Z(\widetilde{G})$ under the map $\widetilde{\Ad}$ is contained in the following kernel $\mu_{2}$:
\begin{equation}
1\,\longrightarrow\, \mu_2\,\longrightarrow\, \Spin_{\dim \gfrak}
\,\longrightarrow\, \SO_{\dim \gfrak}\,\longrightarrow\, 1.
\end{equation}

The embedding $\ad$ in \eqref{eqn:ad} gives an embedding of the corresponding untwisted Lie algebra 
\begin{eqnarray*}
&\widehat{\ad}\,:\, \widehat{\gfrak}\,\longrightarrow\,\widehat{\sfrak \ofrak}(\gfrak)\\
&X\otimes f \,\longmapsto\, \ad(X)\otimes f\,\ \text{ and }\,\ 
c \,\longmapsto\, h^{\vee}(\gfrak)c .
\end{eqnarray*}

Let us briefly recall that an element of the center $ Z(\widetilde{G})$ of the simply connected group can be 
identified with the subgroup of diagram automorphisms of the affine Dynkin diagram of $\widehat{\mathfrak{g}}$. 
This induces a bijection of the set $P_{\ell}(\mathfrak{g})$ of level $\ell$ weights of 
$\mathfrak{g}$.

The non-trivial central element $\sigma\,\in\, \mu_2$ acts on the level one weights of the affine
Lie algebra $\widehat{\sfrak \ofrak}(\gfrak)$ by interchanging the zero-th fundamental weight
$\Lambda_0$ with the first fundamental weight $\Lambda_1$. Thus we have the following elementary fact:

\begin{proposition}\label{prop:branching}
The highest weight integrable moduli $\mathcal{H}_{\omega_0,d_{\ad}}(\gfrak)$ appears with
multiplicity one in the branching of the level one module $\mathcal{H}_{\Lambda_0,1}
(\sfrak \ofrak(\gfrak))$. Let $\Sigma$ be a non-trivial element in $Z(\widetilde{G})$, and let
$\widetilde{\Ad}(\Sigma)\,\in\,\mu_2$ be its image under the adjoint representation. Then $\mathcal{H}_{\Sigma\ast 
\omega_0,d_{\ad}}(\gfrak)$ appears with multiplicity one in the branching of 
$\mathcal{H}_{\widetilde{\Ad}(\Sigma)\bullet\Lambda_{0},1}(\sfrak\ofrak(\gfrak))$. Here $\ast$ 
(respectively, $\bullet$) denotes the action of the center on level $d_{\ad}$ (respectively, $1$) 
weights arising from diagram automorphisms.
\end{proposition}

We refer the reader to Kac--Wakimoto \cite[p.~214, Eq.~4.2.13]{KacWakimoto:88} for a complete description of 
the branching rule for this conformal embedding.

\section{Conformal blocks and adjoint representation} In this section, we first recall the notion of conformal blocks and then analyze the functoriality of conformal blocks under the adjoint representation. 
We use the identification between conformal blocks and the space of non-abelian $G$ theta functions, i.e global section of natural line bundles on $\mathcal{M}_G$ to study the image of $\mathcal{M}_G$ under the adjoint representation $\operatorname{Ad}: \mathcal{M}_G \longrightarrow \mathcal{M}_{\SO(\dim \mathfrak{g})}$. 
\subsection{Conformal blocks}

Let $\pi\,:\, \mathscr{C}\,\longrightarrow\, B$ be a family of stable $n$-pointed curves satisfying
the following conditions:
\begin{itemize}
\item There are disjoint sections $s_i\,:\, B \,\longrightarrow\, \mathscr{C}$, $1\,\leq\, i\,
\leq\, n$, of the family $\pi$, marking smooth points in the fiber $\mathscr{C}_b$.

\item $\mathscr{C} \backslash \sqcup_{i=1}^n s_i(B)$ is affine.

\item There are isomorphisms $\widehat{\mathcal{O}}_{\mathscr{C}/s_i(B)}
\,\cong\, \mathcal{O}_B[[\xi_i]]$, where $\xi_1,\, \cdots\, \xi_n$ are formal parameters.
\end{itemize}
We will denote the above family along with a choice of formal parameters by $\mathfrak{X}$. Let 
$S\,:=\,\sqcup_{j=1}^n s_i(B)$.

For any choice of $n$-tuple $\vec{\lambda}\,=\,(\lambda_1,\,\cdots,\, \lambda_n)$ of level $\ell$ weights of $\gfrak$, one can associate
the quasi-coherent sheaf of covacua 
$$
\mathcal{V}_{\vec{\lambda}}(\mathfrak{X},\gfrak,\ell)\,\,:=\,\,
\frac{ \left(\mathcal{H}_{\vec{\lambda}}\otimes_{\mathbb{C}}\mathcal{O}_B\right)}{\gfrak \otimes
\pi_*(\mathcal{O}_\mathscr{C}(*S))\cdot \left(\mathcal{H}_{\vec{\lambda}}\otimes_{\mathbb{C}}\mathcal{O}_B\right)},
$$
where $\mathcal{H}_{\vec{\lambda}}\,:=\,\mathcal{H}_{\lambda_1,\ell}(\gfrak)\otimes \dots \otimes \mathcal{H}_{\lambda_n,\ell}(\gfrak)$.
Here the action of $\gfrak \otimes \pi_*(\mathcal{O}_\mathscr{C}(*S))$ on $\mathcal{H}_{\vec{\lambda}}\otimes \mathcal{O}_B$ is given by 
Laurent expansion using the formal parameters $\xi_j$. This sheaf of covacua was first constructed in the
work of Tsuchiya--Ueno--Yamada \cite{TUY:89}. 
 The dual of the sheaf of covacua is known as the sheaf of conformal blocks and
it is denoted by $\mathcal{V}^{\dagger}_{\vec{\lambda}}(\mathfrak{X},\gfrak,\ell)$.

We will recall some of it basic properties; the reader is referred to \cite{TUY:89} for a proof.

\begin{theorem}
The sheaf of covacua enjoys the following properties:
\begin{itemize}
\item (Local freeness): The sheaf $\mathcal{V}_{\vec{\lambda}}(\mathfrak{X},\,\gfrak,\,\ell)$ is locally free of finite rank.
Its rank is given by the Verlinde formula.

\item If $\pi\,:\, \mathscr{C}\,\longrightarrow\, B$ is a family of smooth curves, then the sheaf 
$\mathcal{V}_{\vec{\lambda}}(\mathfrak{X},\,\gfrak,\,\ell)$ carries a flat projective connection constructed out of the Segal--Sugawara 
action of $L_{n,\ell}^{\gfrak}$ on integrable modules.

\item (Propagation of vacua): Let $s$ be a new section for a family of nodal curve $\pi\,:\, \mathscr{C}\,\longrightarrow\, S$
disjoint from $s_1,\,\cdots,\, s_n$, and let $\mathfrak{X}'$ denote the new data. Then there is an isomorphism 
$$
\mathcal{V}_{\vec{\lambda}}(\mathfrak{X},\,\gfrak,\,\ell)\,\,\cong\,\, \mathcal{V}_{\vec{\lambda},\Lambda_0}(\mathfrak{X}',\,\gfrak,\,\ell),
$$
where $\Lambda_0$ is the vacuum representation at level $\ell$. 

\item(Factorization theorem): Let $\pi\,:\,\mathcal{C}\,\longrightarrow\, S$ be a family of
nodal curves and $$s_1,\,\dots,\,s_n,\, q_1,\,q_2$$
are disjoint sections of it marking smooth points, and let $\mathcal{D}\,\longrightarrow\, S$ be the family obtained by gluing
$\mathcal{C}$ along $q_1$ and $q_2$. Then there is an isomorphism
$$
\bigoplus \iota_{\mu}\,:\,\mathcal{V}_{\vec{\lambda}}(\mathfrak{X},\,\gfrak,\,\ell)\,\cong\,
\bigoplus_{\mu\in P_{\ell}(\gfrak)}\mathcal{V}_{\vec{\lambda},\mu,\mu^{\dagger}}(\mathfrak{X}',\,\gfrak,\,\ell),
$$
where $\mathfrak{X}$ (respectively, $\mathfrak{X}'$) is associated to the family $\mathcal{D}\,\longrightarrow\, S$
(respectively, $\mathcal{C}\,\longrightarrow\, S$).
\end{itemize}
\end{theorem}

\subsection{Functoriality for conformal blocks}

Let $\varphi\,:\,\gfrak_1 \,\longrightarrow\, \gfrak$ be a homomorphism of simple Lie algebras; this induces a homomorphism
$\varphi\,:\,\widehat{\gfrak}_1 \,\longrightarrow \,\widehat{\gfrak}$ between the affine Lie algebras.
Let $\vec{\Lambda}\,:=\,(\Lambda_1,\,\cdots, \,\Lambda_n)$ be a choice of level one weights of $\widehat{\gfrak}$, and let
$\vec{\lambda}\,:=\,(\lambda_1,\,\cdots,\, \lambda_n)$ be a choice of level $d_{\varphi}$ weights of $\gfrak_1$ such that 
$$\mathcal{H}_{\lambda_i, d_{\varphi}}(\gfrak_1)\,\,\hookrightarrow\,\,\, \mathcal{H}_{\Lambda_i,1}(\gfrak).$$
By functoriality, we get a homomorphism of the corresponding conformal blocks 
\begin{equation}\label{eqn:rlmap}
\varphi\,:\,\, \mathcal{V}_{\vec{\lambda}}(\mathfrak{X},\gfrak_1,d_{\varphi})\,\longrightarrow\,
\mathcal{V}_{\vec{\Lambda}}(\mathfrak{X},\gfrak,1),
\end{equation}where $\mathfrak{X}$ is the data associated to family of $n$-pointed nodal curves with a choice of $n$-formal parameters. Moreover we have the following: 

\begin{proposition}(\cite[Proposition 5.8]{Be1})\label{prop:flatness}
Assume that the embedding $\varphi$ is conformal (Definition \ref{def:conformal}), and the family $\mathfrak{X}$ is smooth.
Then the functorial map $\varphi\,:\, \mathcal{V}_{\vec{\lambda}}(\mathfrak{X},\gfrak_1,d_{\varphi})\,\longrightarrow\,
\mathcal{V}_{\vec{\Lambda}}(\mathfrak{X},\gfrak,1)$ in \eqref{eqn:rlmap} is flat with respect to the projective
connections. In particular, the rank of $\varphi$ is constant. 
\end{proposition}

\subsection{The case of adjoint representations}

We now restrict ourselves to the special case of the adjoint representation. Using the branching rule
in Proposition \ref{prop:branching} together with the functoriality described above, we get a linear map
\begin{equation}\label{eqn:genad}
\ad\,: \,\mathcal{V}_{\Sigma\ast \omega_0}(C,\,\gfrak,\,d_{\ad})\,\longrightarrow\,
\mathcal{V}_{\widetilde{\Ad}(\Sigma)\bullet\Lambda_{0}}(C,\,\sfrak \ofrak(\gfrak),\,1), 
\end{equation}where $C$ is a smooth one-pointed curve of genus $g$. The following question is natural:

\begin{question}
Is the map $\ad$ in \eqref{eqn:genad} non-zero?
\end{question}

\subsubsection{The case where $\Sigma$ is trivial}
Henceforth, we will restrict to the case where $\Sigma$ is trivial. Therefore,
\eqref{eqn:genad} becomes
\begin{equation}\label{eqn:genad2}
\ad\,:\,\mathcal{V}_{\omega_0}(C,\,\gfrak,\, d_{\ad})\,\longrightarrow\, \mathcal{V}_{\Lambda_0}(C,\, \sfrak \ofrak(\gfrak),\,1).
\end{equation}

\begin{theorem}\label{thm:nonzero}
The map $\ad$ in \eqref{eqn:genad2} is non-zero for any smooth curve $C$ of genus $g$. 
\end{theorem}

\begin{proof}
We first note that it is enough to find a curve such that map is non-zero for that curve. Indeed, this follows 
from the fact that the embedding $\ad$ is conformal and the induced map (also denoted by $\ad$) is flat with 
respect to the projective connections and hence the induced map preserves rank (see Proposition \ref{prop:flatness}).

We proceed by induction on the genus of the curve. In the genus zero case 
$\mathcal{V}_{\omega_0}(\mathbb{P}^1,\, \mathfrak{g},\, d_{\ad})\,\cong\, (V_{\omega_0})^{\gfrak}$, where 
$V_{\omega_0}\,\cong\, \mathbb{C}$ is the trivial representation of $\gfrak$. Similarly 
$\mathcal{V}_{\Lambda_0}(\mathbb{P}^1,\, \sfrak \ofrak (\gfrak),\,1)\,\cong \,\mathbb{C}$. Since the trivial 
$\sfrak \ofrak (\gfrak)$-module restricts to the trivial $\gfrak$ module, the result follows in genus zero by 
taking invariants of the trivial representation.
	 
Now consider a family $\mathcal{X}$ over $\mathbb{C}[[t]]$ such that for $q\,\neq\, 0$, the fibers are all one pointed smooth curve of genus 
$g$ and $X_0$ is a nodal curve with exactly one node. By functoriality, we get a map of sheaves over $\mathbb{C}[[t]]$:
$$ \afrak\dfrak\,\,:\,\,\mathcal{V}^{\dagger}_{\Lambda_0}(\mathcal{X},\,\sfrak \ofrak(\gfrak),\,1)\,\longrightarrow\,
\mathcal{V}^{\dagger}_{\omega_0}(\mathcal{X},\,\gfrak,\,d_{\ad}).$$ We have the following diagram:
\begin{equation}\label{eqn:facto}
\begin{tikzcd}
\bigoplus_{\Lambda \in P_1(\gfrak)}\mathcal{V}^{\dagger}_{\Lambda_0,\Lambda, \Lambda^{\dagger}}(\widetilde{X},\,\sfrak \ofrak (\gfrak),
\,1)\ar[r,"\iota_{\Lambda}"]& \mathcal{V}^{\dagger}_{\Lambda_0}({X}_0,\,\sfrak \ofrak (\gfrak),\,1)\ar[d,"\afrak \dfrak_{t=0}"]&\\
&\mathcal{V}^{\dagger}_{\omega_0}({X}_0,\,\gfrak,\,d_{\ad})\ar[r,"\iota_{\mu}^{-1}"]& \bigoplus_{\Lambda \in P_{d_{\ad}}(\gfrak)}
\mathcal{V}^{\dagger}_{\omega_0,\mu, \mu^{\dagger}}(\widetilde{X},\,\gfrak,\,d_{\ad}).
	 \end{tikzcd}
	 \end{equation}
Let $\alpha_{\Lambda,\mu}$ be the following map obtained by restricting the composition of maps in \eqref{eqn:facto}: 
\begin{equation}\label{eqn:factorres}
\alpha_{\Lambda,\mu}\,: \,\,\mathcal{V}^{\dagger}_{\Lambda_0, \Lambda, \Lambda}(\widetilde{X},\, \sfrak \ofrak (\gfrak),\,1)
\,\longrightarrow\, \mathcal{V}^{\dagger}_{\omega_0,\mu, \mu^{\dagger}}(\widetilde{X},\,\gfrak,\,d_{\ad}).
\end{equation}
Observe that the genus of $\widetilde{X}$ is $g-1$, and taking $\Lambda\,=\,\Lambda_0$ and $\mu\,=\,\omega_0$ it follows from
\cite[Proposition 4.4]{BoysalPauly:10} (see also \cite[Proposition 4.3]{Mukhopadhyay:12} and \cite{Mukhopadhyay:15}) that the
above map $\alpha_{\Lambda_0, \mu}$ is zero if $\mu \,\neq\, \omega_0$. Now we have the following commutative diagram for $\widetilde{X}$:
\begin{equation}\label{el}
\begin{tikzcd}[column sep=large]
\mathcal{V}_{\Lambda_0, \Lambda_0, \Lambda_0}^{\dagger}(\widetilde{X},\, \sfrak \ofrak (\gfrak),\,1) 
\ar[r, "\alpha_{\Lambda_0,\omega_0}"] \ar[d, "\cong"]& \mathcal{V}^{\dagger}_{\omega_0,\omega_0, \omega_0}(\widetilde{X},\,\gfrak,\,d_{\ad}) \ar[d, "\cong"]\\
\mathcal{V}_{\Lambda_0}^{\dagger}(\widetilde{X},\, \sfrak \ofrak (\gfrak),\,1) \ar[r, "\ad"]& \mathcal{V}^{\dagger}_{\omega_0}(\widetilde{X},
\,\gfrak,\,d_{\ad})	 
	 \end{tikzcd}
	 \end{equation}
	 
The vertical isomorphisms in \eqref{el} are given by the propagation of vacua. By the induction hypothesis the 
map $\alpha_{\Lambda_0, \omega_0}$ is non-zero. This applied to \eqref{eqn:facto} yields that 
$$\afrak\dfrak_{t=0}\,\,:\,\,\mathcal{V}^{\dagger}_{\Lambda_0}({X}_0,\,\sfrak \ofrak (\gfrak),\,1) 
\,\,\longrightarrow\,\, \mathcal{V}^{\dagger}_{\omega_0}({X}_0,\,\gfrak,\,d_{\ad})$$ is non-zero. Consequently, 
the proof is completed by using semi-continuity.
\end{proof}

\section{Uniformization and Conformal blocks}

In this section, we recall the uniformization theorems connecting conformal blocks with global 
sections of line bundles on moduli stacks. We then have a reinterpretation of Theorem \ref{thm:nonzero}.

\subsection{Non-abelian theta functions and functoriality}

Let ${G}$ be a connected semi-simple group such that the Lie algebra $\mathfrak g$ is
simple. Let
$\mathcal{M}_{{G}}$ be the moduli stack of 
principal ${G}$--bundles on a smooth projective curve $C$. The connected components of the moduli stack 
$\mathcal{M}_{{G}}$ are parametrized by the fundamental group $\pi_1(G)$. We denote by 
$\mathcal{M}^{\delta}_{{G}}$ the component of $\mathcal{M}_{{G}}$ corresponding to $\delta
\,\in\, \pi_1(G)$. We 
will mostly be interested the {\em neutral component} $\mathcal{M}^0_{G}$ of the moduli stack 
corresponding to the trivial element in the fundamental group (it is
the connected component of $\mathcal{M}_{{G}}$ containing the
trivial principal ${G}$--bundle). Denote by $\widetilde{G}$ the simply 
connected cover of $G$. The natural map $ \varpi_1:\ \widetilde{G}\,\longrightarrow\, G$ induces a map
\begin{equation}
\pi_G\,:\, \mathcal{M}_{\widetilde{G}}\,\longrightarrow\, \mathcal{M}^0_{G}
\end{equation}
The kernel of the above map $\varpi_1$ will be denoted by $A$ which is just the fundamental group of $G$. It is known that $A$ is a
product of finite cyclic groups; denote $J_A\,:=\, H^1(C,\, A)$. Then by Beauville-Laszlo-Sorger \cite[Proposition 1.5]{BLS:98}, we have a
long exact sequence 
$$
0\,\longrightarrow\, J_A \,\longrightarrow\, \operatorname{Pic}(\mathcal{M}^{\delta}_{G})
\,\longrightarrow\, \operatorname{Pic}(\mathcal{M}^{\delta}_{\widetilde{G}})\,\longrightarrow\, 0.
$$

It is known that the Picard group of the moduli stack $\mathcal{M}_{\widetilde{G}}$
is infinitely cyclic \cite{BLS:98, DN, KNR:94}. We denote the ample generator of
$\operatorname{Pic}(\mathcal{M}^{\delta}_{\widetilde{G}})$ by $\mathscr{L}_{\widetilde{G}}$. The space of global sections
$H^0(\mathcal{M}_{\widetilde{G}},\, \mathscr{L}_{\widetilde{G}}^{\otimes \ell})$ is known as the space of non-abelian theta functions. 

Given any nonzero homomorphisms $\varphi'\,:\, \gfrak_1 \,\longrightarrow \,\gfrak$ of simple Lie algebras, consider the
corresponding homomorphism $\varphi''\,:\, \widetilde{G}_1 \,\longrightarrow \,\widetilde{G}$ between the
associated simply connected groups; note that $\varphi'$ is necessarily injective. This
$\varphi''$ induces a map of the corresponding moduli stacks via the associated construction 
\begin{equation}\label{vp}
\varphi\,:\, \mathcal{M}_{\widetilde{G}_1} \,\longrightarrow \,\mathcal{M}_{\widetilde{G}}.
\end{equation}
It follows from \cite[p.~59, Section 5]{KumarNarasimhan:97} that $\varphi^*\mathscr{L}_{\widetilde{G}}\,\cong\,
\mathscr{L}_{\widetilde{G}_1}^{\otimes d_{\varphi}}$ for the map in \eqref{vp}, where $d_{\varphi}$ is the Dynkin index
of the embedding $\varphi''$.
This isomorphism induces a map of the global sections 
\begin{equation}\label{vp2}
\varphi^*\,:\, H^0(\mathcal{M}_{\widetilde{G}},\, \mathscr{L}_{\widetilde{G}})\,\longrightarrow\,
H^0(\mathcal{M}_{\widetilde{G}_1},\, \mathscr{L}^{\otimes d_{\varphi}}_{\widetilde{G}_1})
\end{equation}
Now via the uniformization theorems of Beauville-Laszlo \cite{BeauvilleLaszlo:94}, Kumar-Narasimhan-Ramanathan \cite{KNR:94},
Faltings \cite{Faltings:94}, Laszlo-Sorger \cite{LaszloSorger:97}, we get an isomorphism 
\begin{equation}\label{v3}
H^0(\mathcal{M}_{\widetilde{G}},\, \mathscr{L}^{\otimes \ell}_{\widetilde{G}})\,\,\cong\,\,
\mathcal{V}^{\dagger}_{\Lambda_0}(\mathfrak{X},\, \gfrak,\, \ell),
\end{equation}
where $\mathfrak{X}$ is the data of a smooth curve $C$ with one marked point along with a choice of a formal parameter at
the marked point and $\Lambda_0$ is the vacuum representation at level $\ell$. The following diagram is commutative: 
\begin{equation}
\begin{tikzcd}
 H^0(\mathcal{M}_{\widetilde{G}},\, \mathscr{L}_{\widetilde{G}}) \ar[r, "\varphi"]\ar[d, "\cong"]& H^0(\mathcal{M}_{\widetilde{G}_1},\, \mathscr{L}^{\otimes d_{\varphi}}_{\widetilde{G}_1}) \ar[d, "\cong"]\\
 \mathcal{V}^{\dagger}_{\Lambda_0}(\mathfrak{X},\, \gfrak,\, 1)\ar[r,"\varphi"]& \mathcal{V}^{\dagger}_{\Lambda_0}(\mathfrak{X},
\, \gfrak_1,\, d_{\varphi})
\end{tikzcd}
\end{equation}

\subsection{The adjoint representation}

Recall that a square root of the canonical line bundle $K_C$ of $C$ is called a theta characteristic on $C$. 
The set of theta characteristics of $C$ is a torsor over the group of two torsion points $J_2(C)$ of the 
Jacobian. Now identifying $J_2(C)$ with the dual $\widehat{J_2(C)}$ via the Weil pairing, we see that the space 
of theta characteristics is a torsor for $\widehat{J_2(C)}$.

For every theta characteristic $\kappa$ of $C$, Laszlo-Sorger \cite[p.~517, Section 7.8]{LaszloSorger:97} 
constructed a natural square-root $\mathcal{P}_{\kappa}$ of the determinant of cohomology on 
$\mathcal{M}_{\SO_r}$ --- which is known as the Pfaffian line bundle --- along with a canonical Pfaffian
section $s_{\kappa}$. The divisor corresponding to the Pfaffian section consists
of the following associated bundles:
\begin{equation}\label{s4}
\Xi_{\kappa}\,\,:=\,\,\{ E \,\in\, \mathcal{M}_{\SO_r}\,\big\vert\,\, h^0(E\otimes \kappa)\,\neq\, 0\} \,\subseteq \,\mathcal{M}_{\SO_r}.
\end{equation}

Let us recall the following results \cite{Belkale:12} (for $\mathcal{M}_{\SO_r}^0$) and
\cite[Proposition 3.5]{MukhopadhyayWentworth:17} (for $\mathcal{M}_{\SO_r}^-$) about these
Pfaffian sections $s_{\kappa}$ being non-zero. Consider the decomposition
$$\mathcal{M}_{\SO_r}\,=\,\mathcal{M}_{\SO_r}^{0}\sqcup \mathcal{M}_{\SO_r}^{-},$$ where
$\mathcal{M}_{\SO_r}^-$ parametrizes bundles with non-trivial Stiefel-Whitney class. Then, we get:
\begin{enumerate}
\item If $r$ is even and $\kappa$ is any theta characteristic, then $H^0(\mathcal{M}_{\SO_r}^0,\,
\mathcal{P}_\kappa)$ is one dimensional. 

\item If $r$ is even and $\kappa$ is any theta characteristic, then $H^0(\mathcal{M}_{\SO_r}^-,\,
\mathcal{P}_\kappa)$ is zero dimensional. 

\item If $r$ is odd, then $H^0(\mathcal{M}_{\SO_r}^0,\,\mathcal{P}_{\kappa})$ is one
dimensional if and only if $\kappa$ is even.

\item If $r$ is odd, then $H^0(\mathcal{M}_{\SO_r}^-,\,\mathcal{P}_{\kappa})$ is one dimensional if and only if 
$\kappa$ is odd.

\item Each Pfaffian section $s_{\kappa}$ is projectively flat with respect to the Hitchin connection. 

\item The Pfaffian sections $\{s_{\kappa}\}$ are linearly independent and hence 
\begin{itemize}
\item if $r$ is even, then $\dim_{\mathbb{C}}H^0(\mathcal{M}_{\Spin_r},\,\mathscr{L}_{\Spin_r})\,=\,2^{2g}$, and

\item if $r$ is odd, then $\dim_{\mathbb{C}} H^0(\mathcal{M}_{\Spin_r},\,\mathscr{L}_{\Spin_r})\,=\,2^{g-1}(2^g+1)$.
\end{itemize}
\end{enumerate}
A choice of a theta characteristic induces an action of $J_2(C)$ on $H^0(\mathcal{M}_{\Spin_r },\,
\mathscr{L}_{\Spin_r})$ and consequently there is a decomposition (see \cite[p.~2]{Belkale:12},
\cite[p.~14, Proposition 3.7]{MukhopadhyayWentworth:17})
\begin{equation}
H^0(\mathcal{M}_{\Spin_r },\,\mathscr{L}_{\Spin_r})\,=\,\bigoplus_{\chi \in \widehat{J_2(C)}} H^0(\mathcal{M}_{\SO_r}^0,
\, \mathcal{P}_{\kappa}\otimes \mathcal{L}_{\chi}),
\end{equation}
where $\mathcal{L}_{\chi}$ is the line bundle associated to the character $\chi$. 
Moreover by \cite[Proposition 5.2]{BLS:98}, \cite[Proposition 2.2]{Belkale:12}, \cite[Proposition 3.9]{MukhopadhyayWentworth:17}
we have $\mathcal{P}_{\kappa}\otimes \mathcal{L}_{\chi}\,\cong\, \mathcal{P}_{\kappa'}$, where
$\kappa'$ is just the image of $\kappa$ under the action of $\chi \,\in\, \widehat{J_2(C)}$.

The adjoint representation of any connected semisimple group $G$ gives a homomorphism $$\Ad_G\,:\, G\,
\longrightarrow \,\SO_{\dim \gfrak}$$ whose kernel is finite, and we have the following commutative diagram
\begin{equation}\label{l}
\begin{tikzcd}[column sep=huge]
\widetilde{G} \ar[r, "\widetilde{\Ad}"] \ar[d,"\pi_G"'] \ar[rd,"\Ad_{\widetilde{G}}" description ]& \Spin_{\dim \gfrak} \ar[d,"\pi"]\\
G \ar[r, "\Ad_G"'] & \SO_{\dim \gfrak},
\end{tikzcd} 
\end{equation}
where $\widetilde{G}$ is the simply connected cover of $G$. In particular, when $G\,=\,\widetilde{G}$, we have
the following commutative diagram: 
\begin{equation}
\begin{tikzcd}[column sep=huge]
&\mathcal{M}_{\Spin_{\dim{\gfrak}}}\ar[d,"\pi"]\\
	\mathcal{M}_{\widetilde{G}}\ar[d,"\pi_{G}"] \ar[r,"\Ad_{\widetilde{G}}" description ]\ar[ru,"\widetilde{\Ad}"]& \mathcal{M}^{0}_{\SO_{\dim{\gfrak}}}\\
	\mathcal{M}^0_{G}\ar[ru,"\Ad_G"']
\end{tikzcd}
\end{equation}

The following is a consequence of Theorem \ref{thm:nonzero}.

\begin{corollary}\label{cor:nonzero}
If $\dim{\gfrak}$ is odd (respectively, even), then a choice of an even (respectively, any) theta
characteristic $\kappa$ gives a
non-zero map $$\Ad_{\widetilde{G}}\,:\, H^0(\mathcal{M}_{\SO_{\dim \gfrak}},\, \mathcal{P}_{\kappa}) \,\longrightarrow
\,H^0(\mathcal{M}_{\widetilde{G}},\,\mathscr{L}^{d_{\ad}}_{\widetilde{G}}),$$ where $d_{\ad}$ is the Dynkin index of the adjoint representation of $\mathfrak{g} \hookrightarrow \mathfrak{so}(\mathfrak{g})$. Moreover, this map factors through
$H^0(\mathcal{M}_{G}^{0},\, \mathscr{P}_{\chi})$, where $\mathscr{P}_{\chi}$ is a line bundle
on $\mathcal{M}^0_{G}$ which is pulled back from $\mathcal{P}_{\kappa}$; in other words, there is a commutative diagram
$$
\begin{tikzcd}
 H^0(\mathcal{M}_{\SO_{\dim \gfrak}},\, \mathcal{P}_{\kappa}) \ar[r, "\Ad_{\widetilde{G}}"] \ar[rd] &
H^0(\mathcal{M}_{\widetilde{G}},\,\mathscr{L}^{d_{\ad}}_{\widetilde{G}}) \\ 
 & H^0(\mathcal{M}^0_{{G}},\, \mathscr{P}_{\chi}) \ar[u,hook].
 \end{tikzcd}
$$
\end{corollary}

We note an immediate consequence of Corollary \ref{cor:nonzero}.

\begin{corollary}\label{cor:nonzero2}
If $\dim{\gfrak}$ is odd (respectively, even), let $\kappa$ be an even (respectively, any) theta characteristic.
Then the image of the morphism
$$
\mathcal{M}^0_{{G}}\,\, \longrightarrow\,\, \mathcal{M}_{\SO_{\dim \gfrak}},
$$
given by $\Ad_G$ in \eqref{l}, is not contained in the divisor $\Xi_{\kappa}$ (see \eqref{s4}).
\end{corollary}

\begin{proof}
Under the above assumption we know that $\dim H^0(\mathcal{M}_{\SO_{\dim \gfrak}},\,
\mathcal{P}_{\kappa})\,=\, 1$ \cite{Belkale:12, MukhopadhyayWentworth:17}, and the divisor
for any nonzero section of $\mathcal{P}_{\kappa}$ is $\Xi_{\kappa}$. Therefore,
if the image of the above morphism
$\mathcal{M}^0_{{G}}\,\, \longrightarrow\,\, \mathcal{M}_{\SO_{\dim \gfrak}}$
is contained in $\Xi_{\kappa}$, then the homomorphism 
$$
H^0(\mathcal{M}_{\SO_{\dim \gfrak}},\, \mathcal{P}_{\kappa})\, \longrightarrow\,
H^0(\mathcal{M}^0_{{G}},\, \mathscr{P}_{\chi})
$$
in the statement of Corollary \ref{cor:nonzero} becomes the zero map. But Corollary \ref{cor:nonzero}
says that this map is nonzero. In view of this contradiction we conclude that the image of the morphism
$$
\mathcal{M}^0_{{G}}\,\, \longrightarrow\,\, \mathcal{M}_{\SO_{\dim \gfrak}}
$$
is not contained in the divisor $\Xi_{\kappa}$.
\end{proof}

\section{A natural connection}\label{se5}

The main goal of this section is consider the  non-empty Zariski open substack in $\mathcal{M}_G^0$ given by Corollary \ref{cor:nonzero2} as the set of  algebraic principal $G$--bundles on $C$ such that $\operatorname{Ad}(E_G)$ is not in the divisor $\Xi_{\kappa}$, and show that  every element of this set admits a natural algebraic connection. 

Let $E_G$ be an algebraic principal $G$--bundle on $C$. An algebraic connection on $E_G$ produces an algebraic
connection on any algebraic fiber bundle associated to $E_G$, in particular, an algebraic connection is induced on
the adjoint vector bundle $\text{ad}(E_G)$. Let ${\mathcal C}(E_G)$ and ${\mathcal C}(\text{ad}(E_G))$ be the spaces
of algebraic connections on $E_G$ and $\text{ad}(E_G)$ respectively. Let
$$
\Phi_0\,:\, {\mathcal C}(E_G)\, \longrightarrow\, {\mathcal C}(\text{ad}(E_G))
$$
be the above map.

\begin{lemma}\label{lem1}
There is a natural map
$$
\Phi\,:\, {\mathcal C}({\rm ad}(E_G))\,\, \longrightarrow\, \, {\mathcal C}(E_G)
$$
such that $\Phi\circ\Phi_0\,=\, {\rm Id}_{{\mathcal C}(E_G)}$.
\end{lemma}

\begin{proof}
Consider the adjoint homomorphism ${\mathfrak g}\, \hookrightarrow\, {\mathfrak g}\otimes 
{\mathfrak g}^*\,=\, {\mathfrak g}{\mathfrak l}({\mathfrak g})$
. We have the symmetric bilinear 
form on ${\mathfrak g}{\mathfrak l}({\mathfrak g})$ defined by $A\otimes B\, \longmapsto\, 
\text{trace}(AB)$. Its restriction to $\mathfrak g$ is a constant scalar multiple of the 
Killing form of $\mathfrak g$. Consider the corresponding orthogonal decomposition
$$
{\mathfrak g}{\mathfrak l} ({\mathfrak g})\,=\, {\mathfrak g}\oplus {\mathfrak g}^\perp .
$$
Let
\begin{equation}\label{zz}
P\,\,:\,\, {\mathfrak g}{\mathfrak l} ({\mathfrak g})\,\, \longrightarrow\,\, {\mathfrak g}
\end{equation}
be the projection constructed using the above decomposition of ${\mathfrak g}{\mathfrak l} ({\mathfrak g})$.
The adjoint action of $G$ on $\mathfrak g$ produces an action of $G$ on ${\mathfrak g}\otimes 
{\mathfrak g}^*\,=\, {\mathfrak g}{\mathfrak l}({\mathfrak g})$. The projection $P$ in \eqref{zz} is
a homomorphism of $G$--modules.

Let $E_{\text{GL}({\mathfrak g})}$ be the principal $\text{GL}({\mathfrak g})$--bundle $C$ 
corresponding to the vector bundle $\text{ad}(E_G)$. We note that $E_{\text{GL}({\mathfrak g})}$ is the 
quotient of $E_G\times \text{GL}({\mathfrak g})$ where $(z,\, B)$ is identified with $(zg,\, 
\text{Ad}(g^{-1})B \text{Ad}(g))$ for all $g\, \in\, G$. We have a natural map
$$
\Psi\,:\, E_G\,\longrightarrow\, E_{\text{GL}({\mathfrak g})}
$$
that sends any $z\, \in\, E_G$ to the equivalence class of $(z,\, \text{Id}_{\mathfrak g})$.

Let $\nabla$ be an algebraic connection on $\text{ad}(E_G)$. So $\nabla$ is an algebraic $1$--form on $E_{\text{GL}({\mathfrak g})}$,
with values in ${\mathfrak g}{\mathfrak l} ({\mathfrak g})$, satisfying certain conditions. Therefore,
$P\circ (\Psi^*\nabla)$ is a $\mathfrak g$--valued algebraic $1$--form on $E_G$, where $P$ is the projection in
\eqref{zz}. Using the fact that $P$ in \eqref{zz} is a homomorphism of $G$--modules it is straightforward
to check that $P\circ (\Psi^*\nabla)$ satisfies the two conditions needed
to define a connection on $E_G$. It is also evident that the map
$$
\Phi\,:\, {\mathcal C}(\text{ad}(E_G))\,\, \longrightarrow\, \, {\mathcal C}(E_G)
$$
constructed this way satisfies the condition $\Phi\circ\Phi_0\,=\, {\rm Id}_{{\mathcal C}(E_G)}$.
\end{proof}

Take a theta characteristic $\kappa$ on $C$. We assume that if $\dim{\gfrak}$ is odd,
then $\kappa$ is an even theta characteristic. There is no condition on $\kappa$ when $\dim{\gfrak}$ is even.

Consider the map
$$
F\,\,:\,\, \mathcal{M}^0_{{G}}\,\, \longrightarrow\,\, \mathcal{M}_{\SO_{\dim \gfrak}}
$$
in Corollary \ref{cor:nonzero2} that sends any $E_G$ to $\text{ad}(E_G)$ equipped with the
Killing form on the fibers. Take any principal $G$--bundle $E_G\, \in\, \mathcal{M}^0_{{G}}$ such that
\begin{equation}\label{y2}
F(E_G)\,\, \notin\,\,\Xi_{\kappa}
\end{equation}
(see \eqref{s4}); from Corollary \ref{cor:nonzero2} we know that the locus, in $\mathcal{M}^0_{{G}}$,
of all such $E_G$ is a nonempty Zariski open substack. From \eqref{y2} we know that
\begin{equation}\label{y2n}
H^0(C,\, \text{ad}(E_G)\otimes\kappa)\,=\,0\,
=\,H^0(C,\, \text{ad}(E_G)\otimes\kappa),
\end{equation}
which implies that $\text{ad}(E_G)$ is semistable. This in turn
implies that the principal $G$-bundle $E_G$ is semistable.

Since $H^0(C,\, \text{ad}(E_G)\otimes\kappa)\,=\,0\,
=\,H^1(C,\, \text{ad}(E_G)\otimes\kappa)$, from \cite{BB}, \cite{BH} we know that
$\text{ad}(E_G)$ has a canonical algebraic connection. We will now briefly recall the construction of
this connection on $\text{ad}(E_G)$.

For $i\,=\, 1,\, 2$, let $p_i\, :\, C\times C\, \longrightarrow\, C$ be the natural projections. Let
$$
\Delta\, :=\, \{(x,\, x)\, \in\, C\times C\, \big\vert x\,\,\, \in\, C\}\, \subset\, C\times C
$$
be the reduced diagonal divisor. We will identify $\Delta$ with $C$ using the map $x \, \longmapsto\, (x,\, x)$, where
$x\,\in\, C$. The restriction of $(p^*_1\kappa)\otimes (p^*_2\kappa)$
to $\Delta$ is evidently identified with $K_C$. Also,
the restriction of ${\mathcal O}_{C\times C}(\Delta)$ to $\Delta$ is identified with the tangent bundle $TC$ using 
the Poincar\'e adjunction formula (see \cite[p.~146]{GH}). Therefore, we have the following short exact sequence of coherent
sheaves on $C\times C$:

$$\begin{tikzcd}[column sep=tiny]
0 \ar[r]& (p^*_1(\text{ad}(E_G)\otimes\kappa)\otimes (p^*_2(\text{ad}(E_G))))\otimes\kappa))\arrow[d, phantom, ""{coordinate, name=Z}] \ar[r]&(p^*_1(\text{ad}(E_G)\otimes\kappa)) \otimes  (p^*_2(\text{ad}(E_G)\otimes\kappa)) \otimes{\mathcal O}_{C\times C}(\Delta) \arrow[dl,
rounded corners,
to path={ -- ([xshift=2ex]\tikztostart.east)
	|- (Z) [near end]\tikztonodes
	-| ([xshift=-2ex]\tikztotarget.west)
	-- (\tikztotarget)}] \\
& \ad(E_G)^{\otimes 2} \ar[r]&0
\end{tikzcd}
$$

where $\text{ad}(E_G)^{\otimes 2}$ is supported on $\Delta\,=\, C$. Let
\begin{equation}\label{y1}
\begin{tikzcd}[column sep=tiny]
0\ar[r]  & H^0(C\times C,\, (p^*_1(\text{ad}(E_G)\otimes\kappa))\otimes (p^*_2(\text{ad}(E_G)\otimes\kappa))) \arrow[d]&{}\\
&H^0(C\times C,\,(p^*_1(\text{ad}(E_G)\otimes\kappa))
\otimes (p^*_2(\text{ad}(E_G)\otimes\kappa))\otimes{\mathcal O}_{C\times C}(\Delta))\ar[r]\arrow[d, phantom, ""{coordinate, name=Z1}]& H^0(C,\, \text{ad}(E_G)^{\otimes 2})\arrow[dl,
rounded corners,
to path={ -- ([xshift=2ex]\tikztostart.east)
	|- (Z1) [near end]\tikztonodes
	-| ([xshift=-2ex]\tikztotarget.west)
	-- (\tikztotarget)}]  \\
&H^1(C\times C,\, (p^*_1(\text{ad}(E_G)\otimes\kappa))\otimes (p^*_2(\text{ad}(E_G)\otimes\kappa))) & {}\end{tikzcd}
\end{equation}

Since 
\begin{eqnarray*}
H^m(C\times C,\, (p^*_1(\text{ad}(E_G)\otimes\kappa))\otimes
(p^*_2(\text{ad}(E_G)\otimes\kappa)))\\
= \ \bigoplus_{i=0}^m
H^i(C,\, \text{ad}(E_G)\otimes\kappa)\otimes H^{m-i}(C,\, \text{ad}(E_G)\otimes\kappa)
\end{eqnarray*}
 from
\eqref{y2n} we conclude that
$$
H^m(C\times C,\, (p^*_1(\text{ad}(E_G)\otimes\kappa))\otimes
(p^*_2(\text{ad}(E_G)\otimes\kappa)))\, =\, 0
$$
for all $m\, \geq\, 0$. Consequently, from \eqref{y1} it is deduced that
\begin{equation}\label{y3}
H^0(C\times C,\,(p^*_1(\text{ad}(E_G)\otimes\kappa))
\otimes (p^*_2(\text{ad}(E_G)\otimes\kappa))\otimes{\mathcal O}_{C\times C}(\Delta))
\,=\, H^0(C,\, \text{ad}(E_G)^{\otimes 2}).
\end{equation}
The isomorphism
\begin{equation}\label{y4c}
\text{ad}(E_G)\, \stackrel{\sim}{\longrightarrow}\, \text{ad}(E_G)^*
\end{equation}
given by the fiberwise Killing form on $\text{ad}(E_G)$ produces a section
\begin{equation}\label{y4a}
\gamma\, \in\, H^0(C,\, \text{ad}(E_G)^{\otimes 2}).
\end{equation}
{}From the construction of $\gamma$ it is evident that the composition of homomorphisms
\begin{equation}\label{y4b}
\text{ad}(E_G)\,\stackrel{\sim}{\longrightarrow}\, \text{ad}(E_G)^* \,
\stackrel{\gamma}{\longrightarrow}\, \text{ad}(E_G)
\end{equation}
coincides with the identity map of $\text{ad}(E_G)$, where the first isomorphism
is the one in \eqref{y4c}

Let
\begin{equation}\label{y4}
\widetilde{\Gamma}\, \in\, H^0(C\times C,\,(p^*_1(\text{ad}(E_G)\otimes\kappa))
\otimes (p^*_2(\text{ad}(E_G)\otimes\kappa))\otimes{\mathcal O}_{C\times C}(\Delta))
\end{equation}
be the section taken to $\gamma$ (see \eqref{y4a}) by the isomorphism in \eqref{y3}. On the other hand,
the restriction of $(p^*_1\otimes\kappa)\otimes (p^*_2\otimes\kappa)\otimes
{\mathcal O}_{C\times C}(\Delta)$ to the subscheme $2\Delta\, \subset\, C\times C$
has a canonical trivialization \cite[p.~688, Theorem 2.2]{BR}. Restricting
the section $\widetilde{\Gamma}$ in \eqref{y4} to $2\Delta$, and invoking the
trivialization of the restriction of $(p^*_1\otimes\kappa)\otimes (p^*_2\otimes\kappa)\otimes
{\mathcal O}_{C\times C}(\Delta)$ to $2\Delta\, \subset\, C\times C$, we obtain a section
\begin{equation}\label{y5}
\Gamma\,\, \,\in\, H^0\left(2\Delta,\,((p^*_1\text{ad}(E_G))
\otimes (p^*_2\text{ad}(E_G)))\big\vert_{2\Delta}\right).
\end{equation}
The restriction of $\Gamma$ to $\Delta\, \subset\, 2\Delta$ evidently coincides
with $\gamma$ (see \eqref{y4a}) using the identification of $\Delta$ with $C$.
Since the composition of homomorphisms in \eqref{y4b}
coincides with the identity map of $\text{ad}(E_G)$, it follows that $\Gamma$ defines
an algebraic connection on the vector bundle $\text{ad}(E_G)$.

Now using lemma \ref{lem1}, the above algebraic connection on $\text{ad}(E_G)$
produces an algebraic connection on the principal $G$--bundle $E_G$. Therefore, we have
the following:

\begin{theorem}\label{th1}
Take a theta characteristic $\kappa$ on $C$. We assume that if $\dim{\gfrak}$ is odd,
then $\kappa$ is an even theta characteristic. Take any principal $G$--bundle
$E_G\, \in\, \mathcal{M}^0_{{G}}$ that lies in the nonempty Zariski open subset of
$\mathcal{M}^0_{{G}}$ whose image under the map $F$ in Corollary \ref{cor:nonzero2}
lies in the complement of the divisor $\Xi_{\kappa}\, \subset\, \mathcal{M}_{\SO_{\dim \gfrak}}$.
Then $E_G$ has a natural algebraic connection.
\end{theorem}

\section{Isomorphism of torsors}\label{sec-tor}\label{se6}

Let $Y$ be a smooth complex variety. A torsor on $Y$ for the cotangent bundle $\psi\, :\,\Omega^1_Y\, \longrightarrow\, Y$
is an algebraic fiber bundle
$p\, :\, {\mathcal V}\, \longrightarrow\, Y$ together with an isomorphism
$$
\Phi\, :\, {\mathcal V}\times_Y \Omega^1_Y\, \longrightarrow\, {\mathcal V}\times_Y {\mathcal V}
$$
such that
\begin{enumerate}
\item $\psi\circ p_2\,=\,p \circ\Phi$, where $p_2\, :\, {\mathcal V}\times_Y \Omega^1_Y\, \longrightarrow\,\Omega^1_Y$ is the natural
projection, and

\item $\Psi$ defines an action of the fibers of $\Omega^1_Y$ on the fibers of $\mathcal V$.
\end{enumerate}

Let $M^{rs}_G$ denote the moduli space of regularly stable topologically trivial principal $G$-bundles
on $C$. Recall that the a $G$ bundle is regularly stable if it is stable and it's automorphism group is the center $Z(G)$ of the group $G$.  It is known that $M^{rs}_G$ is the smooth locus \cite{BiHo} of the moduli space of semistable topologically
trivial principal $G$-bundles on $C$ except in the only one case where $G\,=\, {\rm SL}(2, {\mathbb C})$
and $\text{genus}(C)\,=\, 2$.

On $M^{rs}_G$, there are two natural torsors for the cotangent bundle
$\Omega^1_{M^{rs}_G}$ which we will now describe. Note that any $E_G\, \in\, M^{rs}_G$ admits an algebraic
connection \cite{Ra}, \cite{AB}. 

\paragraph{\bf{The first torsor}}Let ${\mathcal C}_G$ denotes the moduli space $G$-connections such that
the underlying principal bundle is in $M^{rs}_G$. In other words, ${\mathcal C}_G$ parametrizes pairs of
the $(E_G,\, D)$, where $E_G\, \in\, M^{rs}_G$ and $D$ is an algebraic connection on $E_G$. Let
$$
\Phi\, :\, {\mathcal C}_G \, \longrightarrow\, M^{rs}_G
$$
be the natural projection that sends any $(E_G,\, D)$ to $E_G$. If a principal $G$-bundle $F_G$
admits an algebraic connection, then the space of all algebraic connections on $F_G$ is an affine space
for $H^0(C,\, \text{ad}(E_G)\otimes K_C)$. Therefore, for the projection $\Phi$, the moduli space
${\mathcal C}_G$ is a torsor over $M^{rs}_G$ for the cotangent bundle $\Omega^1_{M^{rs}_G}$.

\paragraph{\bf{The second  torsor}}To describe the second torsor over $M^{rs}_G$ for the cotangent bundle $\Omega^1_{M^{rs}_G}$, first
recall that
\begin{equation}\label{et}
\text{Pic}(M^{rs}_G)\,=\, {\mathbb Z}\oplus {\rm Tor},
\end{equation}
where ${\rm Tor}$ is a finite abelian group \cite[p.~184, Theorem (a)]{BLS:98}.
Any line bundle $\xi$ of finite order has a canonical integrable algebraic connection. In fact,
if $\xi^{\otimes n}$ is the trivial line bundle, then there is a unique connection on $\xi$ which
induces the trivial connection on the trivial line bundle 

Take a line bundle $L$ on $M^{rs}_G$.
Let $\text{Conn}(L)$ denote the sheaf of algebraic connections on $L$, meaning the space of sections
of $\text{Conn}(L)$ over any open subset $U\, \subset\, M^{rs}_G$ is the space of all algebraic connections
on $L\big\vert_U$. To describe $\text{Conn}(L)$ explicitly, let
$$
0\, \longrightarrow\, {\mathcal O}_{M^{rs}_G} \, \longrightarrow\, \text{At}(L) \, \longrightarrow\,
TM^{rs}_G \, \longrightarrow\,0
$$
be the Atiyah exact sequence for $L$ \cite{At}. Tensoring it with $\Omega^1_{M^{rs}_G}$ we get the
exact sequence
\begin{equation}\label{ep}
0\, \longrightarrow\, \Omega^1_{M^{rs}_G} \, \longrightarrow\, \text{At}(L)\otimes
\Omega^1_{M^{rs}_G} \, \stackrel{\Psi}{\longrightarrow}\,
(TM^{rs}_G)\otimes \Omega^1_{M^{rs}_G} \,=\, \text{End}(TM^{rs}_G) \, \longrightarrow\,0.
\end{equation}
Let $\Psi^{-1}({\rm Id}_{TM^{rs}_G})\, \subset\, \text{At}(L)\otimes\Omega^1_{M^{rs}_G}$ be the inverse
image, under the map $\Psi$ in \eqref{ep}, of the image of the section
$M^{rs}_G\, \longrightarrow\, \text{End}(TM^{rs}_G)$ given by the identity map of
$TM^{rs}_G$. From \eqref{ep} it follows immediately that
$\Psi^{-1}({\rm Id}_{TM^{rs}_G})$ is a torsor over $TM^{rs}_G$ for $\Omega^1_{M^{rs}_G}$.

This $\Omega^1_{M^{rs}_G}$-torsor $\Psi^{-1}({\rm Id}_{TM^{rs}_G})$ is identified with the
$\Omega^1_{M^{rs}_G}$-torsor $\text{Conn}(L)$.

{}From the above observation that any line bundle $\xi$ of finite order has a canonical integrable
algebraic connection it follows immediately that
$$
\text{Conn}(L)\,\,=\,\, \text{Conn}(L\otimes\xi).
$$

Now let $L$ be such that $c_1(L)\, \not=\, 0$ (equivalently, $L$ is \textit{not} of finite order
because of \eqref{et}). We have the following theorem, that was originally shown in \cite{BH} for the group 
$G\,=\, \text{SL}(r, {\mathbb C})$. 

\begin{theorem}
The above two
$\Omega^1_{M^{rs}_G}$-torsors ${\mathcal C}_G$ and $\text{Conn}(L)$ over $M^{rs}_G$ are isomorphic
up to a constant rescaling of the action. This means that there is an algebraic isomorphism
\begin{equation}\label{ei}
I \, :\, \text{Conn}(L) \, \longrightarrow\, {\mathcal C}_G
\end{equation}
of fiber bundles over $M^{rs}_G$, and a nonzero number $c\, \in\, {\mathbb R}$, such that
$$
I(z+v)\,=\, I(z)+cv
$$
for all $z\, \in\, \text{Conn}(L)_y$, $y\, \in\, M^{rs}_G$, and $v\,\in\, (\Omega^1_{M^{rs}_G})_y$.
\end{theorem}
\begin{proof}
In view of Corollary \ref{cor:nonzero2} and Theorem \ref{th1}, the proof of \cite{BH} works for any semisimple 
$G$ whose Lie algebra is simple. Therefore, the above result of \cite{BH} holds for any such $G$. However, from 
the above result on $\text{SL}(r, {\mathbb C})$ it is possible to deduce the same result for all semistable $G$ 
whose Lie algebra $\mathfrak g$ is simple; this will be described below.

Take any semisimple
$G$ whose Lie algebra $\mathfrak g$ is simple. The dimension of $\mathfrak g$ is denoted by $r$. Let
\begin{equation}\label{ei0}
P\, :\, M^{rs}_G\, \longrightarrow\, M_{\text{SL}(r, {\mathbb C})}
\end{equation}
be the finite morphism that sends any $E_G$ to the vector bundle $\text{ad}(E_G)$. Let
\begin{equation}\label{ei2}
P^*I \, :\, P^* \text{Conn}(L) \, \longrightarrow\, P^*{\mathcal C}_{{\rm SL}(r, {\mathbb C})}
\end{equation}
be the pullback of the isomorphism in \eqref{ei} by $P$ in \eqref{ei0}. Note that both $P^* \text{Conn}(L)$
and $P^*{\mathcal C}_{{\rm SL}(r, {\mathbb C})}$ are torsors over $M^{rs}_G$ for the pulled back vector bundle $P^*
\Omega^1_{M_{{\rm SL}(r, {\mathbb C})}}$.

Next we will describe a subbundle of $P^*\Omega^1_{M_{{\rm SL}(r, {\mathbb C})}}$.

Consider the injective homomorphism
$${\mathfrak g}\, \longrightarrow\, \mathfrak{sl}({\mathfrak g})
\,=\, \mathfrak{sl}(r, {\mathbb C})$$
given by the adjoint action of $G$ on $\mathfrak g$. The image of $\mathfrak g$ in
$\mathfrak{sl}({\mathfrak g})$ will also be denoted by $\mathfrak g$. The Killing form
on $\mathfrak{sl}({\mathfrak g})$ restricts to a nonzero constant multiple of the Killing form
on $\mathfrak g$. Let
\begin{equation}\label{ei3}
\mathfrak{sl}({\mathfrak g})\,=\, {\mathfrak g}\oplus {\mathfrak g}^\perp
\end{equation}
be the orthogonal decomposition with respect to the Killing form
on $\mathfrak{sl}({\mathfrak g})$.

The decomposition of $\mathfrak{sl}({\mathfrak g})$ in \eqref{ei3} is a decomposition of $G$-modules.
Therefore, for any principal $G$-bundle $E_G$ on $C$, the decomposition in \eqref{ei3} produces
a decomposition
\begin{equation}\label{ei4}
\text{End}^0(\text{ad}(E_G))\,=\, \text{ad}(E_G)\oplus {\mathcal V}(E_G),
\end{equation}
where $\text{End}^0(\text{ad}(E_G))\, \subset\, \text{End}(\text{ad}(E_G))$ is the subbundle of
co-rank one defined by the endomorphisms of trace zero and 
$$
{\mathcal V}(E_G)\,:=\, E_G({\mathfrak g}^\perp)
$$
is the vector bundle on $C$ associated to $E_G$ for the $G$-module ${\mathfrak g}^\perp$ in \eqref{ei3}.
Let
\begin{equation}\label{ei5}
H^0(C,\, \text{End}^0(\text{ad}(E_G))\otimes K_C)\,=\, H^0(C,\, \text{ad}(E_G)\otimes K_C)\oplus
H^0(C,\, {\mathcal V}(E_G)\otimes K_C)
\end{equation}
be the decomposition corresponding to the decomposition in \eqref{ei4}. Note that for
any $E_G\,\in \,M^{rs}_G$, the fiber of $P^* \Omega_{M_{\text{SL}(r, {\mathbb C})}}$
(see \eqref{ei0}) over $E_G$ is $H^0(C,\, \text{End}^0(\text{ad}(E_G))\otimes K_C)$.

Let ${\mathcal V}$ denote the vector bundle over $M^{rs}_G$ whose fiber over any $E_G\,\in\,
M^{rs}_G$ is $H^0(C,\, {\mathcal V}(E_G)\otimes K_C)$. From \eqref{ei5} we conclude that
$\mathcal V$ is a direct summand of $P^* \Omega_{M_{\text{SL}(r, {\mathbb C})}}$.

Consider the isomorphism $P^*I$ in \eqref{ei2} between the two $P^*\Omega^1_{M_{\text{SL}(r,
{\mathbb C})}}$-torsors
$P^* \text{Conn}(L)$ and $P^*{\mathcal C}_{{\rm SL}(r, {\mathbb C})}$. Quotienting both
$P^*\Omega^1_{M_{\text{SL}(r, {\mathbb C})}}$-torsors $P^* \text{Conn}(L)$ and
$P^*{\mathcal C}_{{\rm SL}(r, {\mathbb C})}$ by the subbundle ${\mathcal V}\, \subset\,
P^*\Omega^1_{M_{\text{SL}(r, {\mathbb C})}}$ we obtain a generalization of the isomorphism
$I$ in \eqref{ei} for any semisimple $G$ whose Lie algebra is simple.
\end{proof}

\section{A decomposition of the tangent bundle}\label{se7}

Take a theta characteristic $\kappa$ on $C$. We assume that if $\dim{\gfrak}$ is odd,
then $\kappa$ is an even theta characteristic. We further assume that $\kappa$ has a section (there is always
such a $\kappa$). Take any
\begin{equation}\label{a1}
\sum_{i=1}^m \mu_i.c_i \,\,\in\,\, \big\vert \kappa\big\vert;
\end{equation}
so the line bundle ${\mathcal O}_C(\sum_{i=1}^m \mu_i.c_i)$, where $\{c_i\}_{i=1}^m$ are distinct points of
$C$ and $\mu_i$ are positive integers, is holomorphically isomorphic to $\kappa$. Note that
$$
\sum_{i=1}^m \mu_i\,=\, g-1,
$$
where $g\,=\, \text{genus}(C)$. Let
\begin{equation}\label{a2}
{\mathcal U}\,\,\subset\, \, M^{rs}_G
\end{equation}
denote the nonempty Zariski open subset that parametrizes all $E_G\, \in\, M^{rs}_G$ such that
\begin{equation}\label{a3}
H^0(C,\, \text{ad}(E_G)\otimes\kappa)\,=\, 0\,=\, H^0(C,\, \text{ad}(E_G)\otimes\kappa);
\end{equation}
from Corollary \ref{cor:nonzero2} we know that $\mathcal U$ is nonempty. (Since
$\chi(\text{ad}(E_G)\otimes\kappa)\,=\, 0$, there is only one condition in \eqref{a3}.)

For any $E_G\, \in\, {\mathcal U}$ (see \eqref{a2}), the vector bundle $\text{ad}(E_G)\otimes
{\mathcal O}_C(\sum_{i=1}^m \mu_i.c_i)$ will be denoted by $\widetilde{\rm ad}(E_G)$ for notational convenience.
Consider the natural short exact sequence of coherent sheaves on $C$
\begin{equation}\label{a3b}
0\,\longrightarrow\, \text{ad}(E_G) \,\longrightarrow\, \widetilde{\rm ad}(E_G) \, \longrightarrow\, 
\sum_{i=1}^m \widetilde{\rm ad}(E_G)\big\vert_{\mu_ic_i} \, \longrightarrow\, 0.
\end{equation}
Let
\begin{equation}\label{a4}
	\begin{tikzcd}[column sep=normal]
H^0(C,\, \widetilde{\rm ad}(E_G))\ar[r] & 
\sum_{i=1}^m \widetilde{\rm ad}(E_G)\big\vert_{\mu_ic_i}
 \arrow[d, phantom, ""{coordinate, name=Z1}] \ar[r,"\Phi"]& H^1(C,\, \text{ad}(E_G))\arrow[dl,
rounded corners,
to path={ -- ([xshift=2ex]\tikztostart.east)
	|- (Z1) [near end]\tikztonodes
	-| ([xshift=-2ex]\tikztotarget.west)
	-- (\tikztotarget)}] \\
 &H^1(C,\, \widetilde{\rm ad}(E_G))
\end{tikzcd}
\end{equation}
be the long exact sequence of cohomologies associated to the short exact sequence of sheaves
in \eqref{a3b}. From \eqref{a3} it follows immediately that the homomorphism $\Phi$ in \eqref{a4}
is an isomorphism. Note that $H^1(C,\, \text{ad}(E_G))$ is the fiber of the tangent bundle
$TM^{rs}_G$ at the point $E_G\, \in\, M^{rs}_G$.

For $1\,\leq\, i\, \leq\, m$, let
\begin{equation}\label{wi}
{\mathcal W}_i\, \longrightarrow\, {\mathcal U}
\end{equation}
be the vector bundle whose fiber over any $E_G\, \in\, M^{rs}_G$ is $\widetilde{\rm ad}(E_G)\big\vert_{\mu_ic_i}$.
Note that there is a unique universal adjoint bundle
\begin{equation}\label{a}
{\mathcal A}\,\, \longrightarrow\,\, C\times M^{rs}_G.
\end{equation}
It should be clarified that there may not be a universal principal $G$-bundle over
$C\times M^{rs}_G$. The vector bundle ${\mathcal W}_i$ in \eqref{wi} is the restriction of
${\mathcal A}\otimes p^*_1 {\mathcal O}_C(\sum_{i=1}^m \mu_i.c_i) $ to $(\mu_i c_i)\times{\mathcal U}\,\subset\, C\times M^{rs}_G$,
where $p_1\, :\, C\times M^{rs}_G \, \longrightarrow\, C$ is the natural projection and
$\mathcal A$ is the vector bundle in \eqref{a}. Since $\Phi$ in
\eqref{a4} is an isomorphism, we have the following decomposition:
\begin{equation}\label{a5}
T{\mathcal U}\,\,=\,\, \bigoplus_{i=1}^{m} {\mathcal W}_i.
\end{equation}

We have the following natural question:

\begin{question}\label{q1}
Take any $1\, \leq\, i\, \leq\, m$. Is the holomorphic distribution
$$
{\mathcal W}_i\,\,\subset\,\, T{\mathcal U}
$$
on $\mathcal U$ (in \eqref{a5}) integrable?
\end{question}

Next assume that all $m_i\,=\, 1$ for all $1\,\leq\, i\, \leq\, m$ (see \eqref{a1}). So we have $m\,=\, g-1$.

Note that the Poincar\'e adjunction formula says that the fiber of ${\mathcal O}_C(c_1+\cdots +c_{g-1})$
over any $c_i$, $1\, \leq\, i\, \leq\, g-1$, is identified with the fiber $T_{c_i}C$ of the tangent bundle
over the point $c_i$ (see \cite[p.~146]{GH}). For $1\, \leq\, i\, \leq\, g-1$, let
\begin{equation}\label{a6}
{\mathcal L}_i\,\, :=\, \, {\mathcal U}\times (T_{c_i}C)^{\otimes 2} \, \longrightarrow\, {\mathcal U}
\end{equation}
be the trivializable line bundle over $\mathcal U$ with fiber $(T_{c_i}C)^{\otimes 2}$. The Killing
form on $\mathfrak g$ produces a fiberwise nondegenerate symmetric pairing
$$
{\mathcal A}\otimes {\mathcal A}\,\, \longrightarrow\,\, {\mathcal O}_{C\times M^{rs}_G},
$$
where $\mathcal A$ is the vector bundle in \eqref{a}.
Using it we have a homomorphism
\begin{equation}\label{a7}
\varphi_i\,\,:\,\, {\mathcal W}_i\otimes {\mathcal W}_i\,\, \longrightarrow\,\, {\mathcal L}_i,
\end{equation}
where ${\mathcal W}_i$ and ${\mathcal L}_i$ are the vector bundles
constructed in \eqref{wi} and \eqref{a6} respectively. The homomorphism $\varphi_i$ in \eqref{a7}
is evidently symmetric and fiberwise nondegenerate. Now using \eqref{a5} we get a holomorphic symmetric
fiberwise nondegenerate bilinear form on $T{\mathcal U}$.

\end{document}